\documentclass{amsart}

\usepackage{amsmath,amsfonts,amssymb,amscd}

\def\Q{{\mathbb Q}}

\def\C{{\mathbb C}}

\def\CC{{\mathcal C}}
\def\cc{{\mathfrak c}}

\def\Gal{\mathrm{Gal}}

\def\Aut{\mathrm{Aut}}

\def\cl{\mathrm{cl}}

\def\supp{\mathrm{supp}}
\def\div{\mathrm{div}}

\def\fchar{\mathrm{char}}

\def\dim{\mathrm{dim}}

                             \def\cl{\mathrm{cl}}
                             \def\supp{\mathrm{supp}}

     \def\W{{\mathfrak W}}

                          \def\rr{{\mathfrak r}}
                          \def\RR{{\mathfrak R}}
\def\a{{\mathfrak a}}
                               \def\b{{\mathfrak b}}

\newtheorem{thm}{Theorem}[section]
\newtheorem{lem}[thm]{Lemma}
\newtheorem{cor}[thm]{Corollary}

\theoremstyle{definition}

\newtheorem{ex}[thm]{Example}
\newtheorem{exs}[thm]{Examples}
\newtheorem{sect}[thm]{}

           \newtheorem{rem}[thm]{Remark}

\begin{document}\date{}

\title{Division by $2$ on odd degree hyperelliptic curves and their jacobians}
\author {Yuri G. Zarhin}
\dedicatory{In memory of V.A. Iskovskikh}

\address{Pennsylvania State University, Department of Mathematics, University Park, PA 16802, USA}

\email{zarhin@math.psu.edu}
\thanks{Partially supported by  Simons Foundation  Collaboration grant  \# 585711. \newline
I've started to write this paper
 during my stay in May-June 2016  at the Max-Planck-Institut f\"ur Mathematik (Bonn, Germany) and finished it during my next visit to the Institute in May-July 2018. The MPIM hospitality and support are gratefully acknowledged.}


\begin{abstract}
Let $K$ be an algebraically closed field of characteristic different from 2, $g$ a positive integer, $f(x)$ a degree $(2g+1)$ polynomial with coefficients in $K$ and without multiple roots, 
$\CC:y^2=f(x)$ the corresponding genus $g$ hyperelliptic curve over K, and $J$ the jacobian of $\CC$. We identify $\CC$ with the image of its canonical embedding into $J$ (the infinite point of $\CC$ goes to the identity element of $J$). 
It is well known that for each $\b \in J(K)$ there are exactly $2^{2g}$ elements $\a \in J(K)$ such that $2\a=\b$.
M. Stoll  constructed an  {\sl algorithm} that provides Mumford representations of all  such $\a$, in terms of the Mumford representation of $\b$. The aim of this paper is to give {\sl explicit formulas}  for Mumford representations   of all  such $\a$, when $\b\in J(K)$ is given by
$P=(a,b) \in \CC(K)\subset J(K)$ in terms of coordinates $a,b$. We also prove that if $g>1$ then $\CC(K)$ does {\sl not} contain torsion points with order between $3$ and $2g$.
\end{abstract}

\maketitle

\section{Introduction}

Let $K$ be an algebraically closed field of characteristic different from $2$.

 Let $g \ge 1$ be an integer.
Let $\CC$ be the smooth projective model of the smooth affine plane $K$-curve
$$y^2=f(x)=\prod_{i=1}^{2g+1}(x-\alpha_i)$$
where $\alpha_1,\dots, \alpha_{2g+1}$ are {\sl distinct} elements of $K$. It is well known that $\CC$ is a genus $g$ hyperelliptic curve over $K$ with precisely one {\sl infinite} point, which we denote by $\infty$.  In other words,
$$\CC(K)=\{(a,b)\in K^2\mid  b^2=\prod_{i=1}^{2g+1}(a-\alpha_i)\}\bigsqcup \{\infty\} .$$
Clearly, $x$ and $y$ are nonconstant rational functions on $\CC$, whose only pole is $\infty$. More precisely, the polar divisor of $x$ is $2 (\infty)$ and the polar divisor of $y$ is $(2g+1)(\infty)$. The zero divisor of $y$ is $\sum_{i=1}^{2g+1} (\W_i)$
where
$$\W_i=(\alpha_i,0) \in \CC(K) \ \text{ for all } \ i=1, \dots , 2g, 2g+1.$$
 We write $\iota$ for the hyperelliptic involution 
$$\iota: \CC \to \CC,  (x,y)\mapsto (x,-y), \ \infty \mapsto\infty.$$
The set  of all fixed points of $\iota$ consists of $\infty$ and all $\W_i$.
It is well known that for each $P \in \CC(K)$ the divisor $(P)+\iota(P)-2(\infty)$ is principal. More precisely, if $P=(a,b)\in \CC(K)$ then $(P)+\iota(P)-2(\infty)$ is the divisor of the rational function $x-a$ on $C$. If $D$ is a divisor on $\CC$ then we write $\supp(D)$ for its {\sl support}, which is a finite subset of $\CC(K)$.

We write $J$ for the jacobian of $\CC$, which is a $g$-dimensional abelian variety over $K$.  If $D$ is a degree zero divisor on $\CC$ then we write $\cl(D)$ for its linear equivalence class, which is viewed as an element of $J(K)$. Elements of $J(K)$ may be described in terms of so called {\bf Mumford representations} (see  \cite[Sect. 3.12]{Mumford}, \cite[Sect. 13.2, pp. 411--415, especially, Prop. 13.4, Th. 13.5 and Th. 13.7]{Wash} and Section \ref{divisors} below.)

We will identify $\CC$ with its image in $J$ with respect to the canonical regular map $\CC \hookrightarrow J$ under which  $\infty$ goes to 
the identity element of $J$. In other words, a point $P \in \CC(K)$ is identified with  $\cl((P)-(\infty))\in J(K)$. Then the action of  $\iota$ on  $\CC(K)\subset J(K)$ coincides with multiplication by $-1$ on $J(K)$.
In particular, the list of points of order  2 on $\CC$ consists of  all $\W_i$.

Since $K$ is algebraically closed, the commutative group $J(K)$ is divisible. It is well known that for each $\b \in J(K)$ there are exactly $2^{2g}$ elements $\a =\frac{1}{2}\b\in J(K)$ such that $2\a=\b$.
M. Stoll \cite[Sect. 5]{Stoll} constructed an {\sl algorithm} that provides Mumford representations of all  such $\a$ in terms of the Mumford representation of $\b$. The aim of this paper is to give {\sl explicit formulas} (Theorem \ref{main}) for Mumford representations   of all  $\frac{1}{2}\b$ when $\b\in J(K)$ is given by
$$P=(a,b) \in \CC(K)\subset J(K)$$ 
on $\CC$, in terms of its coordinates $a,b \in K$. (Here $b^2=f(a)$.) The case
$$\b=\infty=0\in J(K)$$ boils down to a well known description of points of order 2 on the jacobian \cite[Ch. 3a, Sect. 2]{Mumford}; one may easily write down explicitly Mumford representations for the order 2 points, see Examples \ref{Exam}.

The paper is organized as follows. In Section \ref{divisors} we recall basic facts about Mumford representations and obtain auxiliary results about divisors on hyperelliptic curves. 
In particular, we prove (Theorem \ref{notheta}) that if $g>1$ then the only point of $\CC(K)$ that is divisible by two in the {\sl theta divisor} $\Theta$ of $J$  (rather than in $J(K)$) is $\infty$. We also prove that $\CC(K)$ does {\sl not} contain points of order $n$ if $3\le n\le 2g$.  In addition, we discuss torsion points on certain natural subvarieties of $\Theta$ when $J$ has ``large monodromy''.
In Section \ref{division} we describe  explicitly for a given $P=(a,b)\in \CC(K)$ the Mumford representation of $2^{2g}$ divisor classes $\cl(D-g(\infty))$ such that $D$ is an effective degree $g$ reduced divisor on $\CC$ and
$$2\cl(D-g(\infty))=P \in \CC(K)\subset J(K).$$
The description is given in terms of collections of  square roots $r_i=\sqrt{a-\alpha_i}$ ($1\le i\le 2g+1$), whose product $\prod_{i=1}^{2g+1}r_i$ is $-b$. (There are exactly $2^{2g}$ choices of such collections of square roots.)

This paper is a follow up of \cite{BZ} where the (more elementary) case of elliptic curves is discussed.  (See also \cite{Schaefer,Yelton}.)

{\bf Acknowledgements}.  I am grateful to Bjorn Poonen, Michael Stoll and Tatiana Bandman for useful comments. My special thanks go to the referee for thoughtful reading of the manuscript.

\section{Divisors on hyperelliptic curves}
\label{divisors}
As usual, a {\sl monic} polynomial is a polynomial with leading coefficient $1$.

Recall \cite[Sect. 13.2, p. 411]{Wash} that if $D$ is an effective divisor of (nonnegative) degree $m$, whose support does {\sl not} contain $\infty$, then the degree zero divisor $D-m(\infty)$ is called {\sl semi-reduced} if it enjoys the following properties.

\begin{itemize}
\item
If $\W_i$ lies in $\supp(D)$ then it appears in $D$ with multiplicity 1.
\item
If a   point $Q$ of $\CC(K)$ lies in $\supp(D)$ and does not coincide with any of $\W_i$  then $\iota(Q)$ does {\sl not} lie in $\supp(D)$.
\end{itemize}
If, in addition, $m \le g$ then $D-m(\infty)$ is called {\sl reduced}.

Notice that a point of $\CC(K)$ that is {\sl not} one of $\W_i$'s may appear in in a (semi)-reduced divisor with multiplicity $>1$.

It is known (\cite[Ch. 3a]{Mumford}, \cite[Sect. 13.2, Prop. 3.6 on p. 413]{Wash}) that for each $\a \in J(K)$ there exist {\sl exactly one}  nonnegative  $m$ and  (effective) degree $m$ divisor  $D$  such that the degree zero divisor $D-m(\infty)$ is {\sl reduced} and  $\cl(D-m(\infty))=\a$.  (E.g., the zero divisor with $m=0$ corresponds to $\a=0$.)  If 
$$m\ge 1, \ D=\sum_{j=1}^m(Q_j)\ \mathrm{ where }  \ Q_j=(a_j,b_j) \in \CC(K) \ \text{ for all } \  j=1, \dots , m$$
(here $Q_j$ do {\sl not} have to be distinct)
then the corresponding
$$\a=\cl(D-m(\infty))=\sum_{j=1}^m Q_j \in J(K).$$
The {\sl Mumford representation} (\cite[Sect. 3.12]{Mumford}, \cite[Sect. 13.2, pp. 411--415, especially, Prop. 13.4, Th. 13.5 and Th. 13.7]{Wash} 
of $\a \in J(K)$ 
is the pair $(U(x),V(x))$ of polynomials $U(x),V(x)\in K[x]$ that enjoys the following properties.
\begin{itemize}
\item
$$U(x)=\prod_{j=1}^m(x-a_j)$$
is a degree $m$ monic polynomial;
\item
 $V(x)$ has degree $<m=\deg(U)$;
 \item
  the polynomial  $V(x)^2-f(x)$ is divisible by $U(x)$;
\item
 each $Q_j$ is a zero of $y-V(x)$, i.e.,
$$b_j=V(a_j), \ Q_j=(a_j,V(a_j))\in \CC(K) \  \text{ for all }  \ j=1, \dots m.$$
\end{itemize}

  Such a pair always exists,  is unique, and (as we have just seen) uniquely determines not only $\a$ but also divisors $D$  and $D-m(\infty)$.

\begin{exs}
\label{Exam}
\begin{itemize}
\item[(i)]
The case $\a=0$ corresponds to $m=0, D=0$ and the pair $(U(x)=1, V(x)=0)$.
\item[(ii)]
The case 
$$\a=P=(a,b) \in \CC(K)\subset J(K)$$
 corresponds to $m=1, D=(P)$
and the pair $(U(x)=x-a, V(x)=b)$.
\item[(iii)]
Let $m \le g$ be a positive integer and $I$ an $m$-element subset of the $(2g+1)$-element set $\{1, \dots, 2g, 2g+1\}$ of positive integers.  Let us consider a degree $m$ effective divisor
$$D_{m,I}=\sum_{i\in I}(\W_i)$$
on $\CC$.  Then the degree zero divisor $D_{m,I}-m(\infty)$ is reduced and its linear equivalence class $\a_{m,I}:=\cl(D_{m,I}-m(\infty))$ has order 2 in $J(K)$, because
$$2 \cl(D_{m,I}-m(\infty))=\cl\left(\left(\sum_{i\in I}2(\W_i)\right)-2m(\infty)\right)=\div(\prod_{i\in I}(x-\alpha_i).$$
Let us consider the polynomials
$$U(x)=U_{m,I}(x):=\prod_{i\in I}(x-\alpha_i), \ V(x)=V_{m,I}(x):=0.$$
Since $f(x)=\prod_{i=1}^{2g+1}(x-\alpha_i)$ is obviously divisible by $U_{m,I}(x)$,
$$f(x)-V_{m,I}(x)^2=f(x)-0^2=f(x)$$
is  divisible by $U_{m,I}(x)$. It follows that $(U_{m,I}(x),0)$ is the Mumford representation of $\a_{m,I}$, since $\W_i=(\alpha_i,0)$ for all $i$.

Clearly, distinct pairs $(m,I)$ correspond to distinct points $\a_{m,I}$. Notice that the number of all $(m,I)$'s equals $2^{2g}-1$ (one has to subtract $1$, because we exclude $m=0$ and empty $I$). At  the same time, $2^{2g}-1$ is the number of elements of order $2$ in $J(K)$. This implies that every order 2 point in $J(K)$ is of the form for exactly one $(m,I)$. Thus, we obtain the Mumford representations for all nonzero halves of zero in $J(K)$.
\end{itemize}

\end{exs}
 
Conversely, if $U(x)$ is a  monic polynomial of degree $m\le g$ and $V(x)$ a polynomial such that $\deg(V)<\deg(U)$ and $V(x)^2-f(x)$ is divisible by $U(x)$, then there exists exactly one  $\a=\cl(D-m(\infty))$ where $D-m(\infty)$ is a reduced divisor, such that $(U(x),V(x))$ is the Mumford representation  of $\a$.

Let $P=(a,b)\in\CC(K)$, i.e.,
$$a,b \in K, \ b^2=f(a)=\prod_{i=1}^n(a-\alpha_i).$$
Recall that our goal is to divide explicitly $P$  by $2$ in $J(K)$, i.e., to give explicit formulas for the {\sl Mumford representation} 
 of {\sl all} $2^{2g}$  divisor classes $\cl(D-m(\infty))$  (with reduced $D-m(\infty)$) such that 
$2D-2m(\infty)$ is linearly equivalent to $(P)-(\infty)$, i.e., the divisor
$2D+\iota(P)$ is linearly equivalent to  $(2m+1)(\infty)$.
(It turns out that each such $D$ has degree $g$ and its support does not contain any of   $\W_i$.)

The following assertion is a simple but useful  exercise in Riemann-Roch spaces (see  Example 4.13 in \cite{StollHC}).

\begin{lem}
\label{key}
Let $D$ be an effective divisor on $\CC$ of  degree $m>0$ such that $m \le 2g+1$ and $\supp(D)$ does not contain $\infty$. Assume that the divisor $D-m(\infty)$ is principal.

\begin{enumerate}
\item[(1)]
Suppose that $m$ is odd.
 Then:
 
 \begin{itemize}
 \item[(i)]
  $m=2g+1$ and there exists exactly one polynomial $v(x)\in K[x]$ such that  the divisor of $y-v(x)$ coincides with $D-(2g+1)(\infty)$. In addition, $\deg(v)\le g$.
    \item[(ii)]
    If $\W_i$ lies in  $\supp(D)$ then it appears in $D$ with multiplicity 1.
    \item[(iii)]
    If $b$ is a nonzero element of $K$ and  $P=(a,b) \in \CC(K)$ lies in  $\supp(D)$ then $\iota(P)=(a,-b)$ does not lie in  $\supp(D)$.
  \end{itemize}
  \item[(2)]
  Suppose that $m=2d$ is even. Then there exists exactly one monic degree $d$ polynomial $u(x)\in K[x]$ such that  the divisor of $u(x)$ coincides with $D-m(\infty)$.  In particular, every point $Q \in \CC(K)$ appears in $D-m(\infty)$ with the same multiplicity as $\iota(Q)$.
  \end{enumerate}
\end{lem}

\begin{proof}
Let $h$ be a rational function on $\CC$, whose divisor coincides with $D-m(\infty)$. Since $\infty$ is the only pole of $h$, the function $h$ is a polynomial in $x,y$ and therefore may be presented as
$h=s(x)y-v(x)$ with  $s,v \in K[x]$.
If $s=0$ then $h$ has at $\infty$ the pole of even order $2\deg(v)$ and therefore $m=2\deg(v)$. 

Suppose that
 $s \ne 0$. Clearly, $s(x)y$ has at $\infty$ the pole of odd order $2\deg(s)+(2g+1) \ge (2g+1)$. So, the orders of the pole for $s(x)y$ and $v(x)$ are distinct, because they have different parity. Therefore the order $m$ of the pole of $h=s(x)y-v(x)$ coincides with
$\max(2\deg(s)+(2g+1), 2\deg(v))\ge 2g+1$.
This implies that $m=2g+1$; in particular, $m$ is odd. It follows that $m$ is {\sl even} if and only if $s(x)=0$, i.e., $h=-v(x)$; in addition, $\deg(v)\le (2g+1)/2$, i.e., $\deg(v)\le g$.  In order to finish the proof of (2), it suffices to divide $-v(x)$ by its leading coefficient and denote the ratio by $u(x)$. (The uniqueness of monic $u(x)$ is obvious.)

Let us prove (1).   
Since $m$ is {\sl odd},
$$m=2\deg(s)+(2g+1)>2\deg(v).$$
 Since $m \le 2g+1$,  we obtain that $\deg(s)=0$, i.e., $s$ is a nonzero element of $K$ and $2\deg(v)< 2g+1$.  The latter inequality means that $\deg(v)\le g$. Dividing $h$ by the constant $s$, we may and will assume that $s=1$ and therefore $h=y-v(x)$ with 
 $$v(x)\in K[x], \ \deg(v) \le g.$$
  This proves (i). (The uniqueness of $v$ is obvious.)
  The assertion (ii)  is contained in Proposition 13.2(b) on  pp. 409-10 of \cite{Wash}.   
 In order to prove (iii),  we just follow arguments on p. 410 of \cite{Wash} (where it is actually proven). Notice that our $P=(a,b)$ is a zero of $y-v(x)$, i.e. $b-v(a)=0$. Since, $b\ne 0$, $v(a)=b \ne 0$ and $y-v(x)$ takes on at $\iota(P)=(a,-b)$ the  value $-b-v(a)=-2b \ne 0$.  This implies that $\iota(P)$ is {\sl not} a zero of $y-v(x)$, i.e., $\iota(P)$ does not lie in $\supp(D)$.
\end{proof}

\begin{rem}
 Lemma \ref{key}(1)(ii,iii) asserts that if $m$ is odd then
the divisor $D-m(\infty)$ is {\sl semi-reduced}. See \cite[the penultimate paragraph on p. 411]{Wash}.  
\end{rem} 

\begin{cor}
\label{bytwo}
Let $P=(a,b)$ be a $K$-point on $\CC$ and $D$ an effective divisor on $\CC$ such that $m=\deg(D)\le g$ and $\supp(D)$ does not contain $\infty$. Suppose that the degree zero divisor $2D+\iota(P)-(2m+1)(\infty)$ is principal. Then:
\begin{itemize}
\item[(i)]
$m=g$ and there exists a polynomial $v_D(x)\in K[x]$ such that $\deg(v_D)\le g$ and the divisor of $y-v_D(x)$ coincides with $2D+\iota(P)-(2g+1)(\infty)$. In particular, $-b=v_D(a)$.
\item[(ii)]
If a point $Q$ lies in  $\supp(D)$ then $\iota(Q)$ does not lie in  $\supp(D)$. In particular, 
\begin{enumerate}
\item
none of $\W_i$ lies in  $\supp(D)$;
\item
$D-g(\infty)$ is reduced.
\end{enumerate}
\item[(iii)]
The point $P$ does not lie in  $\supp(D)$.
\end{itemize}
\end{cor}

\begin{proof}
One has only to apply Lemma \ref{key} to the divisor $2D+\iota(P)$ of {\sl odd} degree $2m+1\le 2g+1$ and notice that $\iota(P)=(a,-b)$ is a zero of $y-v(x)$ while $\iota(\W_i)=\W_i$ for all $i=1, \dots , 2g+1$.
\end{proof}

Let $d \le g$ be a positive integer and $\Theta_d \subset J$ be the image of the regular map
$$\CC^{d} \to J, \ (Q_1, \dots , Q_{d}) \mapsto \sum_{i=1}^{d} Q_i\subset J.$$
It is well known that $\Theta_d$ is an irreducible closed $d$-dimensional subvariety of $J$ that coincides with $\CC$ for $d=1$ and with $J$ if $d =g$; in addition, $\Theta_d\subset\Theta_{d+1}$ for all $d<g$. Clearly, each $\Theta_d$ is stable under multiplication by $-1$ in $J$.
We write $\Theta$ for the $(g-1)$-dimensional {\sl theta divisor} $\Theta_{g-1}$.  

\begin{thm}
\label{notheta}
Suppose that $g>1$ and 
let $$\CC_{1/2}:=2^{-1}\CC \subset J$$
 be the preimage of $\CC$ with respect to multiplication by 2 in $J$.  Then the intersection of $\CC_{1/2}(K)$ and $\Theta$ 
consists of points of order dividing  $2$ on $J$. In particular, the intersection of $\CC$ and $\CC_{1/2}$ consists of $\infty$ and all $\W_i$'s.  In other words,
$$\CC \bigcap 2\cdot\Theta=\{0\}.$$
\end{thm}

\begin{rem}
The case $g=2$ of Theorem \ref{notheta} was done in \cite[Prop. 1.5]{Box1}
\end{rem}

\begin{proof}[Proof of Theorem \ref{notheta}]
Suppose that $m \le g-1$ is a positive integer and we have $m$  (not necessarily distinct) points $Q_1, \dots Q_m$  of $\CC(K)$ and  a point $P\in \CC(K)$ such that in $J(K)$
$$2\sum_{j=1}^m Q_j=P.$$
We need to prove that $P=\infty$, i.e.,  it is the zero of group law in $J$ and therefore
$\sum_{j=1}^m Q_j$ is an element of order 2 (or 1) in $J(K)$. Suppose that this is not true.  Decreasing $m$ if necessary, we may and will assume that none of $Q_j$ is $\infty$ (but $m$ is still positive and does not exceed $g-1$).  Let us consider the effective degree $m$ divisor $D=\sum_{j=1}^m (Q_j)$ on $\CC$. The equality in $J$ means that the divisors $2[D-m(\infty)]$ and $(P)-(\infty)$  on $\CC$ are linearly equivalent.  This means that the divisor $2D+(\iota(P))-(2m+1)(\infty)$ is {\sl principal}. Now Corollary \ref{bytwo} tells us that $m=g$, which is not the case. The obtained contradiction proves that the intersection of $\CC_{1/2}$ and $\Theta$ consists of points of order 2 and 1.  

Since $g>1$,  $\CC\subset \Theta$ and therefore the intersection of $\CC$ and $\CC_{1/2}$ also consists of points of order 2 or 1, i.e., lies in the union of $\infty$ and all $\W_i$'s.  Conversely, since each $\W_i$  has order  $2$ in $J(K)$ and $\infty$ has order 1, they all lie in $\CC_{1/2}$ (and, of course, in $\CC$).
\end{proof}

\begin{rem}
It is known \cite[Ch. VI, last paragraph of Sect. 11, p. 122]{Serre} that the curve $\CC_{1/2}$ is irreducible. (Its projectiveness and smoothness   follow readily from  the  projectiveness 
of $J$ and $\CC$,
the smoothness of $\CC$ and the \'etaleness of multiplication by 2 in $J$.) See \cite{Flynn} for an explicit description of equations that cut out $\CC_{1/2}$ in a projective space.
\end{rem}

\begin{cor}
\label{smallO}
Suppose that $g>1$. Let $m$  be an integer such that $3 \le m \le 2g$.
Then $\CC(K)$ does not contain a point of order $m$ in $J(K)$. In particular, $\CC(K)$  does not contain  points of order $3$ or $4$.
\end{cor}

\begin{rem}
The case $g=2$ of Corollary \ref{smallO} was done in \cite[Prop. 2.1]{Box1}
\end{rem}

\begin{proof}[Proof of Corollary \ref{smallO}]
Suppose that such a point say, $P$ does exists.  Clearly, $P$ is neither $\infty$ nor one of $\W_i$, i.e., $P \ne \iota(P)$.
Let us consider the effective degree $m$ divisor $D=m(P)$. Then the divisor $D-m(\infty)$ is principal and its support contains $P$ but does {\sl not} contain $\iota(P)$.

If $m$ is odd then the desired result follows from Lemma \ref{key}(1). Assume that $m$ is even.  By Lemma \ref{key}(2), the support of $D-m(\infty)$ must contain $\iota(P)$, since it contains $P$. This gives us a contradiction that ends the proof.
\end{proof}

\begin{ex}
\label{order2gp1}
Let us assume that $\fchar(K)$ does {\sl not} divide $(2g+1)$. Then for every nonzero $b \in K$ the monic degree $(2g+1)$ polynomial $x^{2g+1}+b^2$ has no multiple roots and  the point $P=(0,b)$ of the genus $g$ hyperelliptic curve
$$\mathcal{C}: y^2=x^{2g+1}+b^2$$
has order $(2g+1)$ on the jacobian $J$ of $\mathcal{C}$. Indeed, the polar divisor of rational function $y-b$ is $(2g+1)(\infty)$ while $P$ is its only {\sl zero}. Since the degree of $\div(y-b)$ is $0$,
$$\div(y-b)=(2g+1)(P)-(2g+1)(\infty)=(2g+1)((P)-(\infty)).$$
This means that the $K$-point
$$P \in \mathcal{C}(K)\subset J(K)$$
has finite order $m$ that divides $2g+1$. Clearly, $m$ is neither 1 nor 2 (since $P \ne \infty$ and $y(P)=b\ne 0$), i.e., $m\ge 3$. If  $m < (2g+1)$ then  $m\le 2g$ and we get a contradiction to Corollary \ref{smallO}. This proves that the order of $P$ is $(2g+1)$.

Notice that  odd degree genus $2$ hyperelliptic curves with points of order $5=2\times 2+1$ are classified in \cite{Box2}.
\end{ex}

\begin{rem}
If $\fchar(K)=0$ and $g>1$ then the famous theorem of M. Raynaud (conjectured by Yu.I. Manin and D. Mumford) asserts that  an arbitrary genus $g$ smooth projective curve over $K$ embedded into its jacobian contains only finitely many torsion points \cite{Raynaud}.  
\end{rem}

The aim of the rest of this section is to obtain an information about torsion points 
on certain subvarieties $\Theta_d$ when $\CC$ has ``large monodromy''. In what follows we use the notation $[?]$ for the {\sl lower}
integral part of a real number $[?]$.

 Let us start with the following assertion.

\begin{thm}
\label{ThetaD}
Suppose that $g>1$ and 
let  $N$ and $k$ be positive integers such that
$$k <N,  \ N+k \le 2g.$$
Let us put
$$d_{(N+k)}=\left[\frac{2g}{N+k}\right].$$
 Let $K_0$ be a subfield of $K$ such that $f(x)\in K_0[x]$.
Let $\a\in J(K)$ lies on $\Theta_{d_{(N+k)}}$. Suppose that there exists a collection of $k$ (not necessarily distinct) field automorphisms 
$$\{\sigma_1, \dots, \sigma_k\}\subset \Aut(K/K_0)$$
 such that $\sum_{l=1}^k\sigma_l(\a)=N\a$ or $-N\a$.
 Then $\a$ has order 1 or 2 in $J(K)$. 
\end{thm}

\begin{proof}
Clearly, 
$$d_{(N+k)}\le \frac{2g}{N+k}\le \frac{2g}{2+1}<g; \ (N+k)\cdot d_{(N+k)}\le 2g<2g+1.$$
Let us assume that  $2\a\ne 0$ in $J(K)$. We need  to arrive to a contradiction.
There are a positive integer $r\le d_{(N+k)}<g$ and a sequence of points $P_1, \dots, P_r$ of $\CC(K)\setminus\{\infty\}$ such that the linear equivalence class of
 $\tilde{D}:=\sum_{j=1}^r(P_j)-r(\infty)$ equals $\a$.
 We may assume that $r$ is the smallest positive integer that enjoys this property for given $\a$. Then the divisor  $\tilde{D}$ is {\sl reduced}. 
 Indeed, if $\tilde{D}$ is {\sl not} reduced then $r\ge 2$ and we may assume without loss of generality that (say) $P_r=\iota(P_{r-1})$, i.e., the divisor $(P_{r-1})+(P_{r})-2(\infty)$ is principal. Since  $\a \ne 0$,  $r >2$ and therefore $\tilde{D}$ is linear equivalent to
 $$\tilde{D}-\left((P_{r-1})+(P_{r})-2(\infty)\right)=   \sum_{j=1}^{r-2}(P_j)-(r-2)(\infty).$$ This contradicts to the minimality of $r$, and this contradiction  proves that $\tilde{D}$ is  reduced. 
 
 
 We may assume that
   (say) $P_1$ does {\sl not} coincides with 
any of $\W_i$ (here we use the assumption that $2\a \ne 0$); we may also assume that $P_1$ has the largest {\sl multiplicity} in $\tilde{D}$
  among $\{P_1, \dots, P_r\}$; let us denote this multiplicity by  $M$. Since $\tilde{D}$ is reduced, none of $P_j$'s coincides with $\iota P_1$.
The divisor $\sigma_l(\tilde{D})=\sum_{j=1}^r(\sigma_l P_j)-r(\infty)$ is also reduced and its linear equivalence class equals  $\sigma_l\a$ for all $l \in \{1, \dots, k\}$. In particular,  
the multiplicity of each $\sigma_l P_j$ in  $\sigma_l(\tilde{D})$ does {\sl not} exceed $M$; similarly, the multiplicity of each $\iota\sigma_l P_j$ in  $\iota\sigma_l(\tilde{D})$ also does {\sl not} exceed $M$ for every $l$. This implies that if $P$ is any point of $C(K)\setminus\{\infty\}$ that does {\sl not} lie in the support of $\tilde{D}$ then its multiplicity in
$N\tilde{D}+\iota\left(\sum_{l=1}^k\sigma_l(\tilde{D})\right) $ is a nonnegative integer that does {\sl not} exceed $kM$; in addition, the multiplicity of $P$
in $N\tilde{D}+\sum_{l=1}^k\sigma_l(\tilde{D})$ is also a nonnegative integer that also  does {\sl not} exceed $kM$. Notice also that $P_1$ lies in the supports of both
$N\tilde{D}+\iota\left(\sum_{l=1}^k\sigma_l(\tilde{D})\right)$ and $N\tilde{D}+\left(\sum_{l=1}^k\sigma_l(\tilde{D})\right)$ and its multiplicities (in both cases) are, at least, $NM$.

Suppose that $\sum_{l=1}^k\sigma_l(\a)=N\a$. Then the divisor
$$N\tilde{D}+\iota\left(\sum_{l=1}^k\sigma_l(\tilde{D})\right) =N\left(\sum_{j=1}^r(P_j)\right)+\sum_{l=1}^k\left(\sum_{j=1}^r(\iota\sigma_lP_j)\right)-r(N+k)(\infty)$$
is a {\sl principal divisor} on $\CC$. Since 
$$m:=r(N+k) \le (N+k)\cdot d_{(N+k)}\le 2g<2g+1,$$
 we are in position to apply Lemma \ref{key}, which tells us right away that $m$ is {\sl even} 
and there is a monic polynomial $u(x)$ of degree $m/2$, whose divisor coincides with $N\tilde{D}+\iota\sum_{l=1}^k\sigma_l(\tilde{D})$. This implies that   any point $Q\in \CC(K)\setminus\{\infty\}$ appears in $N\tilde{D}+\iota\left(\sum_{l=1}^k\sigma_l(\tilde{D})\right)$ with the same (nonnegative) multiplicity as $\iota{Q}$. It follows that $Q=\iota{P_1}$ appears in  $N\tilde{D}+\iota\left(\sum_{l=1}^k\sigma_l(\tilde{D})\right)$ with the same multiplicity as $P_1$.
 On the other hand, since $\iota{P_1}$ does {\sl not} appear in $\tilde{D}$, its multiplicity in $N\tilde{D}+\iota\sum_{l=1}^k\sigma_l(\tilde{D})$ does {\sl not} exceed $kM$.
Since the multiplicity of $P_1$ in $N\tilde{D}+\iota\left(\sum_{l=1}^k\sigma_l(\tilde{D})\right)$ is, at least, $NM$, 
we conclude that $NM\le kM$, which is not the case, since $k<N$.
This gives us the desired contradiction. 

If $\sum_{l=1}^k\sigma_l(\a)=-N\a$ then literally the same arguments applied to  the principal divisor
$$N\tilde{D}+\sum_{l=1}^k\sigma_l(\tilde{D}) =N\left(\sum_{j=1}^r(P_j)\right)+\sum_{l=1}^k\left(\sum_{j=1}^r(\sigma_l P_j)\right)-r(N+k)(\infty)$$
also lead to the contradiction.
\end{proof}


\begin{sect}
Let $K_0$ be a subfield of $K$ such that $f(x)\in K_0[x]$ and $\bar{K}_0$ the algebraic closure of $K_0$ in $K$. (E.g., one may take as $K_0$  the field that is generated over the prime subfield of $K$ by all the coefficients of $f(x)$.)
We write $\Gal(K_0)$ for the absolute Galois group
$$\Gal(K_0)=\Aut(\bar{K}_0/K)$$
of $K_0$.  It is well known that all torsion points of $J(K)$ actually lie in $J(\bar{K}_0)$.

 Let us consider the following  Galois properties of torsion points of $J(K)$.

\begin{itemize}
\item[(M3)]
If $\a\in J(\bar{K}_0)$ has finite order that is a power of  $2$ then there exists $\sigma \in \Gal(K_0)$ such that $\sigma(\a)=3\a$.
\item[(M2)]
If $\b\in J(\bar{K}_0)$ has finite order that is odd then there exists $\tau \in \Gal(K_0)$ such that $\tau(\b)=2\b$.
\item[(M)]
Let $\a,\b \in J(\bar{K}_0)$ be points of finite order such that the order of $\a$ is a power of $2$ and the order of $\b$ is odd.
Then there exist $\sigma_1, \sigma_2 \in \Gal(K_0)$ such that
$$\sigma_1(\a)=-\a, \ \sigma_1(\b)=2\b; \ \sigma_2(\a)=5\a, \ \sigma_2(\b)=2\b.$$
\end{itemize}

\begin{thm}
\label{thetaTor}
\begin{itemize}
\item[(i)]
Suppose that $g \ge 2$ and $J$ enjoys the property (M3). Let  us put $$d_{(4)}=[2g/4]=[g/2].$$  Let
 $\a\in J(K)$ be a torsion point  that lies on $\Theta_{d_{(4)}}$. 
 
 If the order of $\a$ is a power of $2$
 then it is either $1$ or $2$.
\item[(ii)]
Suppose that $g \ge 2$ and $J$ enjoys the property (M2). Let  us put $$d_{(3)}=[2g/3].$$  Let
 $\b\in J(K)$ be a torsion point of  odd order that lies on $\Theta_{d_{(3)}}$. 
 
 Then $\b=0 \in J(K)$.
\item[(iii)]
Suppose that $g \ge 3$ and $J$ enjoys the property (M). Let  us put $$d_{(6)}=[2g/6]=[g/3].$$  Let
 $\cc \in J(K)$ be a torsion point  that lies on $\Theta_{d_{(6)}}$. 
 
  Then the order of $\cc$ is either $1$ or $2$.
\end{itemize}
\end{thm}

\begin{rem}
In the case of $g=2$ an analogue of Theorem \ref{thetaTor}(i,ii) was earlier proven in \cite[Cor. 1.6]{Box1}.

\end{rem}

\begin{proof}[Proof of Theorem \ref{thetaTor}]
Since all torsion points of $J(K)$ lie in $J(\bar{K}_0)$,
 we may assume that $K=\bar{K}_0$ and therefore $\Gal(K_0)=\Aut(K/K_0)$.
In the first two cases the assertion follows readily from
 Theorem \ref{ThetaD} with $N=3, k=1$ in the case (i)
and with $N=2,k=1$ in the case (ii).    Let us do the case (iii).
We have $\cc=\a+\b$ where the order of $\a$ is odd and the order of $\b$ is a power of $2$.
There  exist $\sigma_1, \sigma_2 \in \Gal(K_0)=\Aut(K/K_0)$ such that
$$\sigma_1(\a)=-\a, \ \sigma_1(\b)=2\b; \ \sigma_2(\a)=5\a, \ \sigma_2(\b)=2\b.$$
This implies that
$$\sigma_1(\cc)+\sigma_2(\cc)=\sigma_1(\a)+\sigma_1(\b)+\sigma_2(\a)+\sigma_2(\b)=-\a+2\b+5\a+2\b=4(\a+\b)=4\cc,$$
i.e., $\sigma_1(\cc)+\sigma_2(\cc)=4\cc$.
Now the desired result follows from
 Theorem \ref{ThetaD} with $N=4, k=2$.
\end{proof}

\begin{ex}
Suppose that $g>1$ and $K$ is the field $\C$ of complex numbers, $\{\alpha_1, \dots, \alpha_{2g+1}\}$ is a $(2g+1)$-element set of   algebraically independent transendental complex numbers
and $K_0=\Q(\alpha_1, \dots, \alpha_{2g+1})$ where $\Q$ is the field of rational numbers. It follows from results of B. Poonen and M. Stoll \cite[Th. 7.1 and its proof]{PoonenStoll} and J. Yelton
\cite[Th. 1.1 and Prop. 2.2]{Yelton}  that the jacobian $J$ of the  {\sl generic} hyperelliptic curve
$$\CC: y^2=\prod_{i=1}^{2g+1}(x-\alpha_i)$$
enjoys the following properties.

{\sl Let us choose   odd integers  $(2n_1+1)$ and  $(2n_2+1)$ and nonnegative integers  $m_1$ and $m_2$.
Suppose that $\a,\b \in J(\bar{K}_0)$ be points of finite order such that the order of $\a$ is a power of $2$ and the order of $\b$ is odd.
Then there exist $\sigma_1, \sigma_2 \in \Gal(K_0)$ such that}
$$\sigma_1(\a)=(2n_1+1)\a, \ \sigma_1(\b)=2^{m_1}\b; \ \sigma_2(\a)=(2n_2+1)\a, \ \sigma_2(\b)=2^{m_2}\b.$$

Choosing $n_1=1$, we obtain that
 $J$ enjoys the property (M3). Choosing $m_1=1$, we obtain that
 $J$  enjoys the property  (M2). Choosing  
$$n_1=1, n_2=2, \  m_1=m_2=1,$$ 
we obtain that $J$ enjoys the property (M).
It follows from Theorem \ref{thetaTor} that torsion points of $J(\C)$ enjoy the following properties.

\begin{itemize}
\item[(i)]
 Any torsion point $\a \in J(\C)$ that lies on  $\Theta_{[g/2]}$ and has  order that is a power of $2$ actually has order $1$ or $2$.
 \item[(ii)]
 If $\b \in J(\C)$ is a torson point  of odd order that lies on  $\Theta_{[2g/3]}$ then  $\b=0\in  J(\C)$.
 \item[(iii)]
 Let $g \ge 3$.  Then any  torsion point $\cc\in J(\C)$ that lies on  $\Theta_{[g/3]}$  has order 1 or 2.
 \end{itemize}

 Notice that B. Poonen and M. Stoll \cite[Th. 7.1]{PoonenStoll}  proved that  the only complex points of finite order in $J(\C)$ that lie
 on $\CC=\Theta_1$ are points of order 1 or 2.  On the other hand, it is well known 
 that $J$ is a simple
 complex abelian variety. Now a theorem of Raynaud \cite{Raynaud2} implies that the set of torsion  points on the theta divisor $\Theta=\Theta_{g-1}$ 
 (actually, on every proper closed subvariety) of $J$ 
 is finite.
\end{ex}

\end{sect}

\section{Division by 2}
\label{division}
If $n$ and $i$  are positive integers and $\mathbf{r}=\{r_1, \dots, r_n\}$ is a sequence of $n$ elements $r_i\in K$ then we write
$$\mathbf{s}_i(\mathbf{r})=\mathbf{s}_i(r_1, \dots, r_n)\in K$$
for the $i$th basic symmetric function in $r_1, \dots, r_n$. If we put $r_{n+1}=0$ then $\mathbf{s}_i(r_1, \dots, r_n)=\mathbf{s}_i(r_1, \dots, r_n, r_{n+1})$.

Suppose we are given a point
$$P=(a,b) \in \CC(K) \subset J(K).$$
Since $\dim(J)=g$, 
there are exactly $2^{2g}$ points $\a \in J(K)$ such that
$$P=2\a \in J(K).$$
Let us choose such an $\a$.
Then there is exactly one effective divisor 
$$D=D(\a)\eqno(1)$$
 of positive degree $m$ on $\CC$ such that  $\supp(D)$ does {\sl not} contain $\infty$, the divisor $D-m(\infty)$ is reduced, and
 $$m \le g, \ \cl(D-m (\infty))=\a.$$ It follows that the divisor $2D+(\iota(P))-(2m+1)(\infty)$ is {\sl principal} and, thanks to  Corollary \ref{bytwo},  $m=g$ and $\supp(D)$ does {\sl not} contains 
  any of $\W_i$. (In addition, $D-g(\infty)$ is reduced.) 
Then  degree $g$ effective divisor 
 $$D=D(\a)=\sum_{j=1}^{g}(Q_j)\eqno(2)$$
    with
$Q_i=(c_j,d_j)\in \CC(K)$.  Since none of $Q_j$ coincides with any of $\W_i$,
$$c_j \ne \alpha_i \ \forall i,j.$$
By  Corollary \ref{bytwo}, there is a polynomial $v_D(x)$ of degree $ \le g$ such that the degree zero divisor $$2D+(\iota(P))-(2g+1)(\infty)$$ is the divisor of $y-v_D(x)$. Since  $\iota(P)=(a,-b)$  and all $Q_j$'s  are zeros of $y-v_D(x)$,
$$b=-v_D(a), \ d_j=v_D(c_j) \  \text{ for all } \ j=1, \dots , g.$$
It follows from Proposition 13.2 on pp. 409--410 of \cite{Wash} that 
$$\prod_{i=1}^{2g+1}(x-\alpha_i)-v_D(x)^2=f(x)-v_D(x)^2=(x-a)\prod_{j=1}^g (x-c_j)^2.\eqno(3)$$
In particular, $f(x)-v_D(x)^2$ is divisible by
$$u_D(x):=\prod_{j=1}^g (x-c_j).\eqno(4)$$

\begin{rem}
Summing up:
$$D=D(\a)=\sum_{j=1}^{g}(Q_j), \ Q_j=(c_j,v_D(c_j)) \  \text{ for all } \ j=1, \dots , g$$
and the degree $g$ monic polynomial $u_D(x)=\prod_{j=1}^g (x-c_j)$ divides $f(x)-v_D(x)^2$.  
Thus (see  the beginning of Section \ref{divisors}),
the  pair $(u_D,v_D)$ is the Mumford representation of $\a$
if $$\deg(v_D)<g=\deg(u_D).$$ This is not always the case: it may happen that $\deg(v_D)=g=\deg(u_D)$ (see below). However, if we replace $v_D(x)$ by its remainder with respect to the division by $u_D(x)$ then we get the Mumford representation  of $\a$ (see below).
\end{rem}

If in (3) we put $x=\alpha_i$ then we get
$$-v_D(\alpha_i)^2=(\alpha_i-a)\left(\prod_{j=1}^g(\alpha_i-c_j)\right)^2,$$
i.e.,
$$v_D(\alpha_i)^2=(a-\alpha_i)\left(\prod_{j=1}^g(c_j-\alpha_i)\right)^2 \  \text{ for all }  \ i=1, \dots , 2g, 2g+1.$$
Since none of $c_j-\alpha_i$ vanishes, we may define
$$
r_i=r_{i,D}:=\frac{v_D(\alpha_i)}{\prod_{j=1}^g(c_j-\alpha_i)}=(-1)^g  \frac{v_D(\alpha_i)}{u_D(\alpha_i)} \eqno(5)$$
with
$$r_i^2=a-\alpha_i \  \text{ for all }  \ i=1, \dots , 2g+1 \eqno(6)$$
and
$$\alpha_i= a-r_i^2, \  c_j-\alpha_i=r_i^2-a+c_j \  \text{ for all }  \ i=1, \dots ,  2g, 2g+1; j=1, \dots , g.$$
Clearly, all $r_i$'s are {\sl distinct} elements of $K$, because their squares are obviously distinct. (By the same token, $r_{j_1} \ne \pm r_{j_2}$ if $j_1\ne j_2$.) Notice that
$$\prod_{i=1}^{2g+1}r_i= \pm b, \eqno(7)$$
because
$$b^2=\prod_{i=1}^{2g+1}(a-\alpha_i)=\prod_{i=1}^{2g+1} r_i^2.\eqno(8)$$
Now we get
$$r_i=\frac{v_D(a-r_i^2)}{\prod_{j=1}^g(r_i^2-a+c_j)},$$
i.e.,
$$r_i \prod_{j=1}^g(r_i^2-a+c_j)-v_D(a-r_i^2)=0 \  \text{ for all }  \ i=1, \dots 2g, 2g+1.$$
This means that the degree $(2g+1)$ {\sl monic} polynomial (recall that $\deg(v_D)\le g$)
$$h_{\mathbf{r}}(t):=t \prod_{j=1}^g (t^2-a+c_j) -v_D(a-t^2)$$
has $(2g+1)$ {\sl distinct} roots $r_1, \dots, r_{2g+1}$. This means that
$$h_{\mathbf{r}}(t)= \prod_{i=1}^{2g+1}(t-r_i).$$
Clearly, $t \prod_{j=1}^g (t^2-a+c_j)$ coincides with the {\sl odd part} of $h_{\mathbf{r}}(t)$ while $-v_D(a-t^2)$ coincides with the {\sl even part} of $h_{\mathbf{r}}(t)$. In particular, if we put $t=0$ then we get
$$ (-1)^{2g+1}\prod_{i=1}^{2g+1}r_i=-v_D(a)=b,$$
i.e.,
$$\prod_{i=1}^{2g+1}r_i=- b.\eqno(9)$$
Hereafter 
$$\mathbf{r}=\mathbf{r}_D:=(r_1, \dots , r_{2g+1}) \in K^{2g+1}.$$
Since $$\mathbf{s}_i(\mathbf{r})=\mathbf{s}_i(r_1, \dots , r_{2g+1})$$  is the $i$th basic symmetric function in 
$r_1, \dots , r_{2g+1}$, 
$$h_{\mathbf{r}}(t)=t^{2g+1}+\sum_{i=1}^{2g+1} (-1)^{i}\mathbf{s}_i(\mathbf{r}) t^{2g+1-i}=\left[t^{2g+1}+\sum_{i=1}^{2g} (-1)^{i}\mathbf{s}_i(\mathbf{r}) t^{2g+1-i}\right]+b.$$
(Since
$$\mathbf{s}_{2g+1}(\mathbf{r})=\prod_{i=1}^{2g+1} r_i=-b,$$
 the constant term of $h_{\mathbf{r}}(t)$ equals $b$.)
Then
$$t\prod_{j=1}^g (t^2-a+c_j)=t^{2g+1}+\sum_{j=1}^g \mathbf{s}_{2j}(\mathbf{r})t^{2g+1-2j},$$
$$-v_D(a-t^2)=\left[-\sum_{j=1}^{g} \mathbf{s}_{2j-1}(\mathbf{r}) t^{2g-2j+2}\right]+b.$$
It follows that
$$\prod_{j=1}^g (t-a+c_j)=t^g+\sum_{j=1}^g \mathbf{s}_{2j}(\mathbf{r})t^{g-j},$$
$$v_D(a-t)=\sum_{j=1}^{g} \mathbf{s}_{2j-1}(\mathbf{r}) t^{g-j+1}-b.$$
This implies that 
$$v_D(t)=\left[\sum_{j=1}^{g} \mathbf{s}_{2j-1}(\mathbf{r}) (a-t)^{g-j+1}\right]-b.\eqno(10)$$
It is also clear that if we consider the degree $g$ monic polynomial
$$U_{\mathbf{r}}(t):=u_D(t)=\prod_{j=1}^g (t-c_j)$$
then
$$
U_{\mathbf{r}}(t)=(-1)^g \left[(a-t)^g+\sum_{j=1}^g \mathbf{s}_{2j}(\mathbf{r})(a-t)^{g-j}\right]. \eqno(11)
$$
 Recall that $\deg(v_D) \le g$ and notice that the coefficient of $v_D(x)$ at $x^g$ is $(-1)^{g}\mathbf{s}_1(\mathbf{r})$. This implies that  the polynomial
$$V_{\mathbf{r}}(t):=v_D(t)-(-1)^{g}\mathbf{s}_1(\mathbf{r}) U_{\mathbf{r}}(t)=$$
$$\left[\sum_{j=1}^{g} \mathbf{s}_{2j-1}(\mathbf{r}) (a-t)^{g-j+1}\right]-b
-
\mathbf{s}_1(\mathbf{r})\left[(a-t)^g+\sum_{j=1}^g \mathbf{s}_{2j}(\mathbf{r})(a-t)^{g-j}\right]=$$
$$\sum_{j=1}^g \left(\mathbf{s}_{2j+1}(\mathbf{r})-\mathbf{s}_1(\mathbf{r})\mathbf{s}_{2j}(\mathbf{r})\right)(a-t)^{g-j}
\eqno(12)
$$
has degree $<g$, i.e.,  
$$\deg(V_{\mathbf{r}})<\deg(U_{\mathbf{r}})=g.$$
 Clearly, $f(x) - V_{\mathbf{r}}(x)^2$ is still divisible by $U_{\mathbf{r}}(x)$, because $u_D(x)=U_{\mathbf{r}}(x)$ divides both
$f(x)-v_D(x)^2$ and $v_D(x)- V_{\mathbf{r}}(x)$.
On the other hand,
$$d_j=v_D(c_j)=V_{\mathbf{r}}(c_j) \  \text{ for all }  j=1, \dots g,$$
because $U_{\mathbf{r}}(x)$ divides  $v_D(x)- V_{\mathbf{r}}(x)$ and vanishes at all $c_j$. Actually, $\{c_1, \dots, c_g\}$ is the list of all roots (with multiplicities) of $U_{\mathbf{r}}(x)$. So,
$$D=D(\a)=\sum_{j=1}^{g}(Q_j), \ Q_j=(c_j,v_D(c_j))=(c_j,V_{\mathbf{r}}(c_j))  \ \forall j=1, \dots , g.$$
This implies (again via the beginning of Section \ref{divisors})
that  the pair  $(U_{\mathbf{r}}(x), V_{\mathbf{r}}(x))$ is the {\sl Mumford representation}  of  $\cl(D-g (\infty))=\a$. 
So,  the formulas (11) and (12) give us an explicit construction of ($D(\a)$ and) $\a$ in terms of $\mathbf{r}=(r_1, \dots, , r_{2g+1})$ for each of $2^{2g}$ choices of $\a$ with $2\a=P\in J(K)$.  
On the other hand,  in light of (6)-(8), there is exactly the same number $2^{2g}$ of  choices of  collections of square roots $\sqrt{a-\alpha_i}$ ($1\le i \le 2g$) with product  $-b$. Combining it with (9), we obtain  that for each choice
of square roots $\sqrt{a-\alpha_i}$'s  with $\prod_{i=1}^{2g+1}\sqrt{a-\alpha_i}=-b$ there is precisely one $\a \in J(K)$ with $2\a=P$ such that the corresponding $r_i$ defined by (5) coincides with  chosen $\sqrt{a-\alpha_i}$ for all $i=1, \dots , 2g+1$, and the Mumford  representation $(U_{\mathbf{r}}(x), V_{\mathbf{r}}(x))$  for this $\a$ is given by  formulas (11)-(12). This gives us the following assertion.

\begin{thm}
\label{main}
Let $P=(a,b)\in \CC(K)$. Then the $2^{2g}$-element set 
$$M_{1/2,P}:=\{\a \in J(K)\mid 2\a=P\in \CC(K)\subset J(K)\}$$
can be described as follows.  Let $\RR_{1/2,P}$ be the set of all $(2g+1)$-tuples $\rr=(\rr_1, \dots , \rr_{2g+1})$ of elements of $K$ such that
$$\rr_i^2=a-\alpha_i \  \text{ for all }  \ i=1, \dots , 2g, 2g+1;  \ \prod_{i=1}^{2g+1}\rr_i=-b.$$ 
Let $\mathbf{s}_i(\rr)$ be the $i$th basic symmetric function in  $\rr_1, \dots , \rr_{2g+1}$. Let us put
$$
U_{\rr}(x)=(-1)^g \left[(a-x)^g+\sum_{j=1}^g \mathbf{s}_{2j}(\rr)(a-x)^{g-j}\right],$$
$$V_{\rr}(x)=\sum_{j=1}^g \left(\mathbf{s}_{2j+1}(\rr)-\mathbf{s}_1(\rr)\mathbf{s}_{2j}(\rr)\right)(a-x)^{g-j}.$$
Then there is a natural bijection between  $\RR_{1/2,P}$ and $M_{1/2,P}$ such that $\rr \in  \RR_{1/2,P}$  corresponds to $\a_{\rr}\in M_{1/2,P}$ with Mumford representation  $(U_{\rr},V_{\rr})$.  
More explicitly, if $\{c_1, \dots, c_g\}$ is the list of all $g$ roots (with multiplicities) of $U_{\rr}(x)$ then $\rr$ corresponds to
$$\a_{\rr}=\cl(D-g(\infty))\in J(K), \ 2\a_{\rr}=P$$
where the divisor
$$D=D(\a_{\rr})=\sum_{j=1}^g (Q_j), \ Q_j=(c_j,V_{\rr}(c_j))\in \CC(K) \  \text{ for all }  \ j=1, \dots, g.$$
In addition, none of $\alpha_i$ is a root of $U_{\rr}(x)$ (i.e., the polynomials $U_{\rr}(x)$ and $f(x)$ are relatively prime) and
$$\rr_i=\mathbf{s}_1(\rr)+(-1)^g\frac{V_{\rr}(\alpha_i)}{U_{\rr}(\alpha_i)}
 \  \text{ for all }  \ i=1, \dots , 2g, 2g+1.$$
\end{thm}

\begin{proof}
Actually we have already proven all the assertions of Theorem \ref{main} except the last formula for $\rr_i$.
It follows from (4) and (5) that
$$\rr_i=(-1)^g \frac{v_{D(\a_{\rr})}(\alpha_i)}{u_{D(\a_{\rr})}(\alpha_i)}=
(-1)^g \frac{v_{D(\a_{\rr})}(\alpha_i)}{U_{\rr}(\alpha_i)}.$$
It follows from (12) that
$$v_{D(\a_{\rr})}(x)=(-1)^{g}\mathbf{s}_1(\rr) U_{\rr}(x)+V_{\rr}(x).$$
This implies that
$$\rr_i=(-1)^g\frac{(-1)^{g}\mathbf{s}_1(\rr) U_{\rr}(\alpha_i)+V_{\rr}(\alpha_i)}{U_{\rr}(\alpha_i)}=\mathbf{s}_1(\rr)+(-1)^g\frac{V_{\rr}(\alpha_i)}{U_{\rr}(\alpha_i)}.$$
\end{proof}

\begin{cor}
\label{s1r}
We keep the notation and assumptions of Theorem \ref{main}. Then
$$2g\cdot  \mathbf{s}_1(\rr)=(-1)^{g+1}\sum_{i=1}^{2g+1}\frac{V_{\rr}(\alpha_i)}{U_{\rr}(\alpha_i)}.$$
In particular, if $\fchar(K)$ does not divide $g$ then
 $$\mathbf{s}_1(\rr)=\frac{(-1)^{g+1}}{2g} \cdot\sum_{i=1}^{2g+1}\frac{V_{\rr}(\alpha_i)}{U_{\rr}(\alpha_i)}.$$
 On the other hand, if $\fchar(K)$ divides $g$ then
 $$\sum_{i=1}^{2g+1}\frac{V_{\rr}(\alpha_i)}{U_{\rr}(\alpha_i)}=0.$$
\end{cor}

\begin{proof}
It follows from the last assertion of Theorem \ref{main} that
$$\mathbf{s}_1(\rr)=\sum_{i=1}^{2g+1}\rr_i=\sum_{i=1}^{2g+1}\left(\mathbf{s}_1(\rr)+(-1)^g\frac{V_{\rr}(\alpha_i)}{U_{\rr}(\alpha_i)}\right)=$$
$$(2g+1)\mathbf{s}_1(\rr)+(-1)^{g}\sum_{i=1}^{2g+1}\frac{V_{\rr}(\alpha_i)}{U_{\rr}(\alpha_i)}.$$
This implies that
$$0=2g \cdot \mathbf{s}_1(\rr)+(-1)^{g}\sum_{i=1}^{2g+1}\frac{V_{\rr}(\alpha_i)}{U_{\rr}(\alpha_i)},$$
i.e.,
$$2g \cdot\mathbf{s}_1(\rr)=(-1)^{g+1}\sum_{i=1}^{2g+1}\frac{V_{\rr}(\alpha_i)}{U_{\rr}(\alpha_i)}.$$
\end{proof}

\begin{cor}
\label{s1rbis}
We keep the notation and assumptions of Theorem \ref{main}. Let $i,l$ be two distinct integers such that
$$1 \le i, l\le 2g+1.$$
Then
$$\mathbf{s}_1(\rr)=\frac{(-1)^g}{2}\times \frac{\left(\alpha_l +\left(\frac{V_{\rr}(\alpha_l)}{U_{\rr}(\alpha_l)}\right)^2\right)-\left(\alpha_i +\left(\frac{V_{\rr}(\alpha_i)}{U_{\rr}(\alpha_i)}\right)^2\right)}{ \left(\frac{V_{\rr}(\alpha_i)}{U_{\rr}(\alpha_i)}-\frac{V_{\rr}(\alpha_l)}{U_{\rr}(\alpha_l)}\right)}.$$
\end{cor}

\begin{proof}
We have
$$\rr_i=\mathbf{s}_1(\rr)+(-1)^g\frac{V_{\rr}(\alpha_i)}{U_{\rr}(\alpha_i)}, \ \rr_l=\mathbf{s}_1(\rr)+(-1)^g\frac{V_{\rr}(\alpha_l)}{U_{\rr}(\alpha_l)}.$$
Recall that 
$$\rr_i^2=a-\alpha_i \ne a-\alpha_l=\rr_l^2.$$
In particular, 
$$\rr_i\ne \rr_l \  \text{ and therefore } \ \frac{V_{\rr}(\alpha_i)}{U_{\rr}(\alpha_i)}\ne \frac{V_{\rr}(\alpha_l)}{U_{\rr}(\alpha_l)}.$$
We have
$$\alpha_l -\alpha_i=(a-\alpha_i)-(a-\alpha_l)=\rr_i^2- \rr_l^2=$$
$$\left(\mathbf{s}_1(\rr)+(-1)^g\frac{V_{\rr}(\alpha_i)}{U_{\rr}(\alpha_i)}\right)^2-\left(\mathbf{s}_1(\rr)+(-1)^g\frac{V_{\rr}(\alpha_l)}{U_{\rr}(\alpha_l)}\right)^2=$$
$$(-1)^g\cdot 2 \cdot \mathbf{s}_1(\rr)\cdot \left(\frac{V_{\rr}(\alpha_i)}{U_{\rr}(\alpha_i)}-\frac{V_{\rr}(\alpha_l)}{U_{\rr}(\alpha_l)}\right)+
\left(\frac{V_{\rr}(\alpha_i)}{U_{\rr}(\alpha_i)}\right)^2-\left(\frac{V_{\rr}(\alpha_l)}{U_{\rr}(\alpha_l)}\right)^2.$$
This implies that
$$(-1)^g\cdot 2 \cdot \mathbf{s}_1(\rr)\cdot \left(\frac{V_{\rr}(\alpha_i)}{U_{\rr}(\alpha_i)}-\frac{V_{\rr}(\alpha_l)}{U_{\rr}(\alpha_l)}\right)=
\left(\alpha_l +\left(\frac{V_{\rr}(\alpha_l)}{U_{\rr}(\alpha_l)}\right)^2\right)-\left(\alpha_i +\left(\frac{V_{\rr}(\alpha_i)}{U_{\rr}(\alpha_i)}\right)^2\right).$$
This means that
$$\mathbf{s}_1(\rr)=\frac{(-1)^g}{2}\times \frac{\left(\alpha_l +\left(\frac{V_{\rr}(\alpha_l)}{U_{\rr}(\alpha_l)}\right)^2\right)-\left(\alpha_i +\left(\frac{V_{\rr}(\alpha_i)}{U_{\rr}(\alpha_i)}\right)^2\right)}{ \left(\frac{V_{\rr}(\alpha_i)}{U_{\rr}(\alpha_i)}-\frac{V_{\rr}(\alpha_l)}{U_{\rr}(\alpha_l)}\right)}.$$
\end{proof}

\begin{rem}
\label{minusRR}
Let $\rr=(\rr_1, \dots , \rr_{2g+1})\in \RR_{1/2,P}$ with $P=(a,b)$.  Then 
for all $i=1, \dots, 2g, 2g+1$
$$(-\rr_i)^2=\rr_i^2=a-\alpha_i$$
and
$$\prod_{i=1}^{2g+1}(-\rr_i)=(-1)^{2g+1}\prod_{i=1}^{2g+1}\rr_i=-(-b)=b.$$
This means that
 $$-\rr=(-\rr_1, \dots , -\rr_{2g+1})\in \RR_{1/2,\iota(P)}$$
(recall that $\iota(P)=(a, -b)$). It follows from Theorem \ref{main} that
$$U_{-\rr}(x)=U_{\rr}(x), \ V_{-\rr}(x)=-V_{\rr}(x)$$
and therefore $\a_{-\rr}= -\a_{\rr}$.
\end{rem}

\begin{rem}
\label{fromUVtoR}
The last assertion of Theorem \ref{main} combined with  Corollary \ref{s1rbis} allow us to reconstruct explicitly $\rr=(\rr_1, \dots, \rr_{2g+1})$ and 
$P=(a,b)$ if we
are given the polynomials $U_{\rr}(x), V_{\rr}(x)$ (and, of course, $\{\alpha_1, \dots, \ \alpha_{2g+1}\}$).
\end{rem}

\begin{ex}
Let us take as $P=(a,b)$ the point $\W_{2g+1}=(\alpha_{2g+1},0)$.  Then $b=0$ and  $\rr_{2g+1}=0$. We have $2^{2g}$ arbitrary independent choices of (nonzero) square roots $\rr_i=\sqrt{\alpha_{2g+1}-\alpha_i}$ with $1 \le i \le 2g$ (and always get an element of $\RR_{1/2,P}$).  Now Theorem \ref{main} gives us (if we put $a=\alpha_{2g+1},b=0$)  all $2^{2g}$  points $\a_{\rr}$ of order 4 in $J(K)$ with $2\a_{\rr}=\W_{2g+1}$. Namely, let $s_i$ be the $i$th basic symmetric function in  $(\rr_1, \dots, \rr_{2g})$. Then the Mumford representation $(U_{\rr},V_{\rr})$ of $\a_{\rr}$ is given by 
$$
U_{\rr}(x)=(-1)^g \left[(\alpha_{2g+1}-x)^g+\sum_{j=1}^g s_{2j}\cdot (\alpha_{2g+1}-x)^{g-j}\right],$$
$$V_{\rr}(x)=\sum_{j=1}^g \left(s_{2j+1}-s_1 s_{2j}\right)(\alpha_{2g+1}-x)^{g-j}.$$
In particular, if $\alpha_{2g+1}=0$ then
$$\rr_i=\sqrt{-\alpha_i} \ \text{ for all } i=1,  \dots, 2g,$$
$$
U_{\rr}(x)=x^g+\sum_{j=1}^g (-1)^j s_{2j} x^{g-j},$$
$$V_{\rr}(x)=\sum_{j=1}^g \left(s_{2j+1}-s_1 s_{2j}\right)(-x)^{g-j}.$$
\end{ex}

\end{document}